\documentclass[12pt]{amsart}
\usepackage{graphicx, amssymb, amsaddr, amsthm, amsfonts, latexsym, epsfig, amsmath, amscd, amsbsy}

\theoremstyle{definition}
\newtheorem{theorem}{Theorem}[section]
\newtheorem{lemma}[theorem]{Lemma}
\newtheorem{corollary}[theorem]{Corollary}
\newtheorem{proposition}[theorem]{Proposition}

\newtheorem{defi}[theorem]{Definition}
\newtheorem{remark}[theorem]{Remark}

\numberwithin{equation}{section}

\def\ie{{\em i.e.,} }
\def\eg{{\em e.g.} }

\newfont\bbf{msbm10 at 12pt}
\def\eps{\varepsilon}
\def\phi{\varphi}
\def\R{{\mathbb R}}
\def\C{{\mathbb C}}
\def\None{{\mathbb N}}
\def\N0{{\mathbb N}_0}
\def\N{{\mathbb N}}
\def\Z{{\mathbb Z}}

\def\C0{{\mathfrak C}_0}
\def\A0{{\mathfrak A}}

\def\Cl{\mathop\mathrm{Cl}}
\def\mesh{\mathop\mathrm{mesh}}

\def\diam{\mbox{\rm diam}\,}
\def\orb{\mbox{\rm orb}}
\def\theta{\vartheta}
\def\le{\leqslant}

\def\IL{\underleftarrow\lim([0,s/2],T_s)}
\def\ILp{\underleftarrow\lim([0,s'/2],T_{s'})}

\def\chain{{\mathcal C}}

\begin{document}

\title[On isotopy and unimodal inverse limits]{On isotopy and unimodal inverse limit spaces}
\author{H.~Bruin and S.~\v{S}timac}
\address{University of Surrey and University of Zagreb}
\thanks{HB was supported by EPSRC grant EP/F037112/1.
S\v{S} was supported in part by NSF 0604958 and in part by the MZOS Grant
037-0372791-2802 of the Republic of Croatia.}

\subjclass[2000]{54H20, 37B45, 37E05}
\keywords{isotopy, tent map, inverse limit space}

\begin{abstract}
We prove that every self-homeomorphism $h : K_s \to K_s$ on the inverse limit space $K_s$ of tent map $T_s$
with slope $s \in (\sqrt 2, 2]$ is isotopic to a power of the shift-homeomorphism $\sigma^R : K_s \to K_s$.
\end{abstract}

\maketitle

\baselineskip=18pt

\section{Introduction}\label{sec:intro}
The solution of Ingram's Conjecture constitutes a major advancement in the classification of unimodal
inverse limit spaces and the group of self-homeomorphisms on them. This conjecture was posed by Tom
Ingram in 1992 for tent maps $T_s : [0, 1] \to [0, 1]$ with slope $\pm s$, $s \in [1, 2]$, defined as
$T_s(x) = \min\{sx, s(1-x)\}$. The turning point is $c = \frac12$ and we denote its iterates by
$c_n = T^n_s(c)$. The inverse limit space  $K_s = \IL$ consists of the {\em core}
$\underleftarrow\lim([c_2, c_1],T_s)$ and the $0$-composant $\C0$, \ie the composant of the point
$\bar 0 := (\dots, 0,0,0)$, which compactifies on the core of the inverse limit space. Ingram's Conjecture
reads:
\begin{quote}
If $1 \leq s < s' \leq 2$, then the corresponding inverse limit spaces
 $\IL$ and $\ILp$ are non-homeomorphic.
\end{quote}
The first results towards solving this conjecture were obtained for tent maps
with a finite critical orbit \cite{Kail2,Stim,Betal}.
Raines and \v{S}timac \cite{RS} extended these results to tent maps
with a possibly infinite, but non-recurrent critical orbit. Recently Ingram's Conjecture was solved completely
(in the affirmative) in \cite{BBS}, but we still know very little of the structure of inverse limit spaces
(and their subcontinua) for the case that $\orb(c)$ is infinite and recurrent, see
\cite{BBD, BB, Bsubcontinua}.

Given a continuum $K$ and $x \in K$, the {\em composant $A$ of $x$}
is the union of the proper subcontinua of $K$ containing $x$.
For slopes $s \in (\sqrt{2}, 2]$, the core is indecomposable
(\ie it cannot be written as the union of two proper subcontinua),
and in this case we also proved \cite{BBS}
that any self-homeomorphism $h:K_s \to K_s$ is pseudo-isotopic to a power $\sigma^R$ of the
shift-homeomorphism $\sigma$ on the core. This means that $h$ permutes the composants of the core of $K_s$
in the same way as $\sigma^R$ does, and it is a priori a weaker property than isotopy. This is for instance
illustrated by the $\sin \frac1x$-continuum, defined as the graph $\{ (x, \sin \frac1x) : x \in (0,1]\}$
compactified with a {\em bar} $\{ 0 \} \times [-1,1]$. There are homeomorphisms that reverse the orientation
of the bar, and these are always pseudo-isotopic, but never isotopic, to the identity. Since such
$\sin \frac1x$-continua are precisely the non-trivial subcontinua of Fibonacci-like inverse limit spaces
\cite{Bsubcontinua}, this example is very relevant to our paper.

In this paper we make the step from pseudo-isotopy to isotopy. To this end, we exploit so-called {\em folding points}, \ie points in the core of $K_s$ where the local structure of the core of $K_s$ is not
that of a Cantor set cross an arc. In the next section we prove
the following results:

\begin{theorem}\label{thm:folding}
If $s \in (\sqrt2, 2]$, and $h:K_s \to K_s$ is a homeomorphism, then there is $R \in \Z$ such that
$h(x) = \sigma^R(x)$ for every folding point $x$ in $K_s$.
\end{theorem}

Folding points $x = (\dots, x_{-2}, x_{-1}, x_0)$ are characterized by the fact that each entry $x_{-k}$
belongs to the omega-limit set $\omega(c)$ of the turning point $c = \frac12$, see \cite{Rain}. This gives
the immediate corollary for those slopes such that the critical orbit $\orb(c)$ is  dense in $[c_2, c_1]$,
which according to \cite{BM} holds for Lebesgue a.e.\ $s \in [\sqrt{2}, 2]$.

\begin{corollary}\label{cor:dense}
If $\orb(c)$ is dense in $[c_2, c_1]$, then for every homeomorphism $h:K_s \to K_s$ there is $R \in \Z$ such
that  $h = \sigma^R$ on the core of $K_s$.
\end{corollary}

The more difficult case, however, is when $\orb(c)$ is not dense in $[c_2, c_1]$. In this case, $h$ can be at
best isotopic to a power of the shift, because at non-folding points, where the core of $K_s$ is a Cantor set
cross an arc, $h$ can easily act as a local translation. It is shown in \cite{BKRS} that for tent maps with
non-recurrent critical point (or in fact, more generally long-branched tent maps), every homeomorphism
$h:K_s \to K_s$ is indeed isotopic to a power of the shift. The proof exploits the fact that in this case,
so-called $p$-points (indicating folds in the arc-components of $K_s$) are separated from each other, at least in arc-length semi-metric.
Here we prove the general result.

\begin{theorem}\label{thm:Fibo_iso}
If $s \in (\sqrt2, 2]$, and $h:K_s \to K_s$ is a homeomorphism, then there exists $R \in \Z$ such that $h$ is isotopic to $\sigma^R$.
\end{theorem}

The paper is organized as follows. In Section~\ref{sec:folding} we give basic definitions and prove results
on how homeomorphisms act on folding points, \ie Theorem~\ref{thm:folding} and Corollary~\ref{cor:dense}.
These proofs depend largely on the results obtained in \cite{BBS}. In
Section~\ref{sec:homeo} we present the additional arguments needed for the isotopy result and finally prove
Theorem~\ref{thm:Fibo_iso}.

\section{Inverse limit spaces of tent maps and folding points}\label{sec:folding}

Let $\N = \{ 1,2,3,\dots\}$ be the set of natural numbers and $\N_0 = \None \cup \{ 0 \}$. The tent map
$T_s:[0,1] \to [0,1]$ with slope $\pm s$ is defined as $T_s(x) = \min\{sx, s(1-x)\}$. The critical or turning
point is $c = 1/2$ and we write $c_k = T_s^k(c)$, so in particular $c_1 = s/2$ and $c_2 = s(1-s/2)$. Also let
$\orb(c)$ and $\omega(c)$ be the orbit and the omega-limit set of $c$. We will restrict $T_s$ to the interval
$I = [0,s/2]$; this is larger than the {\em core} $[c_2, c_1] = [s-s^2/2, s/2]$, but it contains the fixed
point $0$ on which the $0$-composant $\C0$ is based.

The inverse limit space $K_s = \IL$ is
\[
\{ x = (\dots, x_{-2}, x_{-1}, x_0) :  T_s(x_{i-1}) = x_i \in [0,s/2]
\text{ for all } i \leq 0\},
\]
equipped with metric $d(x,y) = \sum_{n \le 0} 2^n |x_n - y_n|$ and {\em induced} $($or {\em shift$)$
homeo\-morphism}
\[
\sigma(\dots, x_{-2}, x_{-1}, x_0) = (\dots, x_{-2}, x_{-1}, x_0, T_s(x_0)).
\]
Let $\pi_k: \IL \to I$, $\pi_k(x) = x_{-k}$ be the $k$-th projection map. Since $0 \in I$, the endpoint
$\bar 0 := (\dots, 0,0,0)$ is contained in  $\IL$.
The composant of $\IL$ of $\bar 0$ will be denoted as $\C0$;
it is a ray converging to, but disjoint from the core $\underleftarrow\lim([c_2, c_1],T_s)$ of the inverse limit space.
We fix $s \in (\sqrt{2}, 2]$; for these parameters $T_s$ is not
renormalizable and $\underleftarrow\lim([c_2, c_1],T_s)$ is indecomposable.
Moreover, the arc-component of $\bar 0$ coincides with the composant of
$\bar 0$, but for points in the core of $K_s$, we have to make the
distinction between arc-component and composant more carefully.

A point $x = (\dots, x_{-2}, x_{-1}, x_0) \in K_s$ is called a {\em $p$-point} if $x_{-p-l} = c$ for some
$l \in \N_0$. The number $L_p(x) := l$ is the {\em $p$-level} of $x$. In particular,
$x_0 = T_s^{p + l}(c)$. By convention, the endpoint $\bar 0$ of $\C0$ is also a $p$-point and
$L_p(\bar 0) := \infty$, for every $p$. The ordered set of all $p$-points of the composant $\C0$ is denoted
by $E_p$, and the ordered set of all $p$-points of $p$-level $l$ by $E_{p,l}$. Given an arc $A \subset K_s$
with successive $p$-points $x^0, \dots , x^n$, the {\em $p$-folding pattern} of $A$ is the sequence
$$
FP_p(A) := L_p(x^0), \dots , L_p(x^n).
$$
Note that every arc of $\C0$ has only finitely many $p$-points, but an arc $A$ of the core of $K_s$ can have
infinitely many $p$-points. In this case, if $(u^i)_{i \in \Z}$ is the sequence of successive $p$-points
of $A$, then $FP_p(A) = (L_p(u^i))_{i \in \Z}$. The {\em folding pattern of the composant} $\C0$, denoted by
$FP(\C0)$, is the sequence $L_p(z^1), L_p(z^2), \dots $, $L_p(z^n), \dots$, where
$E_p = \{ z^1, z^2, \dots , z^n, \dots \}$ and $p$ is any nonnegative integer. Let $q \in \None$, $q > p$, and
$E_q = \{ y^0, y^1, y^2, \dots \}$. Since $\sigma^{q-p}$ is an order-preserving homeomorphism of $\C0$, it is
easy to see that $\sigma^{q-p}(z^i) = y^i$ for every $i \in \None$, and $L_p(z^i) = L_q(y^i)$. Therefore, the
folding pattern of $\C0$ does not depend on $p$.

\begin{defi}\label{df:salient}
Let $(s_i)_{i \in \None}$ be a sequence of $p$-points of $\C0$ such that $0 \leq L_p(x) < L_p(s_i)$ for every
$p$-point $x \in (\bar 0 , s_i)$. We call $p$-points satisfying this property \emph{salient}.
\end{defi}

Since for every slope $s > 1$ and $p \in \N_0$, the folding pattern of the $0$-composant $\C0$ starts as
$\infty \ 0 \ 1 \ 0 \ 2 \ 0 \ 1 \ \dots$, and since by definition $L_p(s_1) > 0$, we have $L_p(s_1) = 1$.
Also, since $s_i = \sigma^{i-1}(s_1)$, $L_p(s_i) = i$, for every $i \in \None$. Note that the salient
$p$-points depend on $p$: if $p \geq q$, then the salient $p$-point $s_i$ equals the salient $q$-point
$s_{i+p-q}$.

A {\em folding point} is any point $x$ in the core of $K_s$ such that no neighborhood of $x$ in  core of
$K_s$ is homeomorphic to the product of a Cantor set and an arc. In \cite{Rain} it was shown that
$x = (\dots , x_{-2}, x_{-1}, x_0)$ is a folding point if and only if $x_{-k} \in \omega(c)$ for all $k \geq 0$.
We can characterize folding points in terms of $p$-points as follows:

\begin{lemma}\label{lem:folding_p-points}
Let $p$ be arbitrary. A point $x \in K_s$ is a folding point if and only if there is a sequence of $p$-points
$(x^k)_{k \in \N}$ such that $x^k \to x$ and $L_p(x^k) \to \infty$.
\end{lemma}

\begin{proof}
$\Rightarrow$ Take $m \geq p$ arbitrary. Since $\pi_m(x) \in \omega(c)$ there is a sequence of post-critical
points $c_{n_i} \to \pi_m(x)$. This means that any point $y^i = (\dots, c_{n_i}, c_{n_i+1}, \dots , c_{n_i+m})$
is a $p$-point with $p$-level $L_p(y^i) = n_i+m-p$. Furthermore, for each $0 \leq j \leq m$,
$|\pi_j(y^i)-\pi_j(x)| \to 0$ as $i \to \infty$. Since $m$ is arbitrary, we can construct a diagonal sequence
$(x^k)_{k \in \N}$ of $p$-points, by taking a single element from $(y^i)_{i \in \N}$ for each $m$, such that
$\sup_{j \leq k} |\pi_j(x^k) - \pi_j(x)| \to 0$ as $k \to \infty$. This proves that $x^k \to x$ and
$L_p(x^k) \to \infty$.
\\
$\Leftarrow$ Take $m$ arbitrary. Since $x^k \to x$, also
$|\pi_m(x^k) - \pi_m(x)| \to 0$ and $\pi_m(x^k) = c_n$ for $n = L_p(x^k) + p-m$. But $L_p(x^k) \to \infty$, so $\pi_m(x) \in \omega(c)$.
\end{proof}

A continuum is {\em chainable} if for every $\eps > 0$, there is a cover $\{ \ell^1, \dots , \ell^n\}$ of
open sets (called {\em links}) of diameter $< \eps$ such that $\ell^i \cap \ell^j \neq \emptyset$ if and only
if $|i-j| \leq 1$. Such a cover is called a {\em chain}. Clearly the interval $[0,s/2]$ is chainable.

\begin{defi}\label{df:chain}
We call $\chain_p$ a {\em natural} chain of $\IL$ if
\begin{enumerate}
\item there is a chain $\{ I^1_p, I^2_p, \dots , I^n_p \}$ of $[0,s/2]$ such that
$\ell^j_p := \pi_p^{-1}(I^j_p)$ are the links of $\chain_p$;
\item each point $x \in \cup_{i=0}^p T_s^{-i}(c)$ is the boundary point of some link $I^j_p$;
\item for each $i$ there is $j$ such that $T_s(I^i_{p+1}) \subset I^j_p$.
\end{enumerate}
\end{defi}

If $\max_j |I^j_p| < \eps s^{-p}/2$ then $\mesh(\chain_p) := \max\{ \diam(\ell) :\ell \in \chain_p\} < \eps$,
which shows that $\IL$ is indeed chainable.
Condition (3) ensures that $\chain_{p+1}$ \emph{refines} $\chain_p$ (written $\chain_{p+1} \preceq \chain_p$).

We are now ready to prove Theorem~\ref{thm:folding}.

\begin{proof}[Proof of Theorem~\ref{thm:folding}]
Let $h : K_s \to K_s$ be a homeomorphism. Let $x, y \in K_s$ be folding points with $h(x) = y$. For
$i \in \N_0$ let $q_i, p_i \in \N$ be such that for sequences of chains $(\chain_{q_i})_{i \in \N_0}$ and
$(\chain_{p_i})_{i \in \N_0}$ of $K_s$ we have
$$\cdots \prec h(\chain_{q_{i+1}}) \prec \chain_{p_{i+1}} \prec h(\chain_{q_i}) \prec \chain_{p_i} \prec \cdots \prec h(\chain_{q_{1}}) \prec \chain_{p_{1}} \prec h(\chain_{q}) \prec \chain_{p},$$
where $q_0 = q$ and $p_0 = p$. Let $(\ell^x_{q_i})_{i \in \N_0}$ be sequence of links such that
$x \in \ell^x_{q_i} \in \chain_{q_i}$, and similarly for $(\ell^y_{p_i})_{i \in \N_0}$. Then
$\ell^x_{q_{i+1}} \subset \ell^x_{q_i}$, $\ell^y_{p_{i+1}} \subset \ell^y_{p_i}$ and
$h(\ell^x_{q_i}) \subset \ell^y_{p_i}$. Let $(s'_{d_i})_{i \in \N}$ be a sequence of salient $q$-points
with $s'_{d_i} \to x$ as $i \to \infty$. Then for every $i$ there exist $j_i$ such that
$s'_{d_{j_i}} \in \ell^x_{q_i}$,  $h(s'_{d_{j_i}}) \in \ell^y_{p_i}$ and $h(s'_{d_{j_i}}) \to y$ as
$i \to \infty$. By \cite[Theorem 4.1]{BBS} the midpoint of the arc component $A_i$ of $\ell^y_{p_i}$
which contains $h(s'_{d_{j_i}})$ is a salient $p_i$-point $s''_{m_i}$. Since $s''_{m_i}, y \in \ell^y_{p_i}$,
for every $i$ and $\diam \, \ell^y_{p_i} \to 0$ as $i \to \infty$, we have $s''_{m_i} \to y$. Since
$s'_{d_i}$ is a salient $q$-point and $s'_{d_i} \in \ell^x_q$, $s''_{m_i}$ can be also considered as a salient
$p$-point and is also the midpoint of the arc component $B_i \supset A_i$ of $\ell^y_{p}$ which contains
$h(s'_{d_{j_i}})$. Therefore, $s''_{m_i} = s_{d_{j_i}+M}$, where $M$ is as in
\cite[Theorem 4.1]{BBS}.

Let $R = M-q+p$. By \cite[Corollary 5.3]{BBS}, $R$ does not depend on $q$, $p$ and $M$. Since
$\sigma^R : K_s \to K_s$ is a homeomorphism, and since $s'_{d_i} \to x$ as $i \to \infty$, we have
$\sigma^R(s'_{d_i}) \to \sigma^R(x)$ as $i \to \infty$. Note that $\sigma^R(s'_{d_{j_i}}) = s_{d_{j_i}+M}$
and $s_{d_{j_i}+M} \to y$. Therefore $\sigma^R(x) = y$, \ie $\sigma^R(x) = h(x)$.
\end{proof}

\begin{proof}[Proof of Corollary~\ref{cor:dense}]
If $\orb(c)$ is dense in $[c_2, c_1]$, every point $x$ in the core of $K_s$ satisfies
$\pi_k(x) \in \omega(c)$ for all $k \in \N$. By \cite{Rain}, this means that every point
is a folding point, and hence the previous theorem implies that $h \equiv \sigma^R$ on
the core of $K_s$.
\end{proof}

\begin{remark}\label{rem:folding}
A point $x \in K_s$ is an {\em endpoint} of an atriodic continuum, if for every pair of subcontinua
$A$ and $B$ containing $x$, either $A \subset B$ or $B \subset A$. The notion of folding point is more
general than that of end-point. An example of a folding point that is not an endpoint is the midpoint
$x$ of a {\em double spiral} $S$, \ie a continuous image of $\R$ containing a single folding point $x$
and two sequences of $p$-points
$$
\dots y^k \prec y^{k+1} \prec \dots \prec x \prec \dots \prec z^{k+1} \prec z^k \dots
$$
converging to $x$ such that the arc-length $\bar{d}(y^k, y^{k+1}),\ \bar{d}(z^k, z^{k+1}) \to 0$ as
$k \to \infty$. Here $\prec$ denotes the induced order on $S$.

It is natural to classify arc-components $\A0$ according to the folding points they
may contain. For arc-components $\A0$, we have the following possibilities:
\begin{itemize}
\item $\A0$ contains no folding point.
\item $\A0$ contains one folding point $x$, \eg if $x$ is an end-point
of $\A0$ or $\A0$ is a double spiral.
\item $\A0$ contains two folding points, \eg if $\A0$ is the bar of a $\sin \frac1x$-continuum.
\item $\A0$ contains countably many folding points. One can construct tent maps such that the folding points of
its inverse limit space belong to finitely many arc-components that are periodic under $\sigma$, but where there are still countably folding
points.\footnote{An example is the tent-map where $c_1$ has symbolic itinerary
(kneading sequence) $\nu = 100101^201^301^4$ $01^5\dots$. Then the two-sided
itineraries of folding points are limits of $\{\sigma^j(\nu)\}_{j \geq 0}$.
The only such two-sided limit sequences are $1^\infty . 1^\infty$
and $\{\sigma^j(1^\infty . 01^\infty) : j \in \Z\}$.
Since they all have left tail $\dots 1111$, these folding points
belong to the arc-component of the point $(\dots,p,p,p)$ for the fixed
point $p = \frac{s}{1+s}$. This use of two-sided symbolic itineraries
was introduced for inverse limit spaces in \cite{BD}.}
\item $\A0$ contains uncountably many folding points, \eg if
$\omega(c) = [c_2,c_1]$, because then every point in the core is a
folding point.
\end{itemize}

This is clearly only a first step towards a complete classification.
\end{remark}

\begin{defi}\label{df:linksym}
Let $\ell^0, \ell^1, \dots, \ell^k$ be those links in $\chain_p$ that are successively visited by an arc
$A \subset \C0$ (hence $\ell^i \neq \ell^{i+1}$, $\ell^i \cap \ell^{i+1} \neq \emptyset$ and
$\ell^i = \ell^{i+2}$ is possible if $A$ turns in $\ell^{i+1}$). Let $A^i \subset \ell^i$ be the
corresponding arc components such that $\Cl A^i$ are subarcs of $A$.
We call the arc $A$
\begin{itemize}
\item {\em $p$-link symmetric} if $\ell^i = \ell^{k-i}$ for $i = 0, \dots, k$;
\item {\em maximal $p$-link symmetric} if it is  $p$-link symmetric
and there is no $p$-link symmetric arc $B \supset A$ and passing through more links than $A$.
\end{itemize}
The $p$-point of $A^{k/2}$ with the highest $p$-level is called the
{\em center} of $A$, and the link $\ell^{k/2}$ is called the {\em central link} of $A$.
\end{defi}

\section{Isotopic Homeomorphisms of Unimodal Inverse Limits}
\label{sec:homeo}

It is shown in \cite{BBS} that every salient $p$-point $s_l \in \C0$ is the
center of the maximal $p$-link symmetric arc $A_l$.
We denote the central link that $s_l$ belongs to
by $\ell_p^{s_l}$.
For a better understanding of this section, let us mention that
a key idea in \cite{BBS}
is that under a homeomorphism $h$ such that $h(\chain_q) \prec \chain_p$,
(maximal) $q$-link symmetric arcs have to
map to (maximal) $p$-link symmetric arcs, and for this reason
$h(s_m) \in \ell_p^{s_l}$ for some appropriate $m \in \N$ (see \cite[Theorem 4.1]{BBS}).

\begin{lemma}\label{lem:finite}
Let $h : K_s \to K_s$ be a homeomorphism pseudo-isotopic to $\sigma^R$, and let $q, p \in \N_0$ be such that $h(\chain_{q}) \preceq \chain_{p}$.
Let $x$ be a $q$-point in the core of $K_s$ and let
$\ell_p^{s_l} \in \chain_{p}$ be the link containing both $\sigma^p(x)$
and salient $p$-point $s_l$, where $l = L_p(\sigma^R(x))$.
Suppose that the arc-component $W_x$
of $\ell_p^{s_l}$ containing $\sigma^R(x)$ does not contain any folding point. Then $h(x) \in W_x$.
\end{lemma}

\begin{proof}
Since $W_x$ does not contain any folding point, it contains finitely many $p$-points. Note that $W_x$
contains at least one $p$-point since $\sigma^R(x) \in W_x$ is a $p$-point. Since $\C0$ is dense in $K_s$,
there exists a sequence $(W_i)_{i \in \N}$ of arc-components of $\ell_p^{s_l}$ such that  $W_i \subset \C0$,
$FP_p(W_i) = FP_p(W_x)$ for every $i \in \N$, and $W_i \to W_x$ in the Hausdorff metric. Let
$(x_i)_{i \in \N}$ be a sequence of $q$-points such that for every $i \in \N$, $L_q(x_i) = L_q(x)$,
$x_i \to x$ and $\sigma^R(x_i) \in W_i$. Obviously $(x_i)_{i \in \N} \subset \C0$,
$L_p(\sigma^R(x_i)) = L_p(\sigma^R(x))$ and $\sigma^R(x_i) \to \sigma^R(x)$. Since $h$ is a
homeomorphisms, $h(x_i) \to h(x)$. It follows by the construction in the proof of \cite[Proposition 4.2]{BBS}
that $h(x_i) \in W_i$ for every $i \in \N$. Therefore $h(x) \in W_x$.
\end{proof}

\begin{corollary}\label{cor:arc-components}
Let $h : K_s \to K_s$ be a homeomorphism pseudo-isotopic to $\sigma^R$. Then $h$ permutes arc-components
of $K_s$ in the same way as $\sigma^R$.
\end{corollary}

\begin{proof}
Since $h$ is a homeomorphism, $h$ maps arc-components to arc-components. Let $\A0$ be an arc-component of
$K_s$. Let us suppose that $\A0$ contains a folding point, say $x$. Then $h(x) = \sigma^R(x)$ implies
$h(\A0) = \sigma^R(\A0)$.

Let us assume now that $\A0$ does not contain any folding point. There exist $q, p \in \N_0$ such that
$h(\chain_{q}) \preceq \chain_{p}$ and that $h(\A0)$ is not contained in a single link of $\chain_{p}$.
Then $\A0$ is not contained in a single link of $\chain_{q}$. Let $\ell_q \in \chain_q$ and
$V \in \ell_q \cap \A0$ be an arc-component of $\ell_q$ such that $V$ contains at least one $q$-point,
say $x$. Let $\ell_p^{s_l} \in \chain_{p}$ be such that $l = L_p(\sigma^R(x))$. Let $W \subset \ell_p^{s_l}$
be arc-component containing $\sigma^R(x)$. Since $\A0$ does not contain any folding point, $h(\A0)$ does not
contain any folding point implying $W$ does not contain any folding point. Then, by Lemma \ref{lem:finite},
$h(x) \in W$ implying $h(\A0) = \sigma^R(\A0)$.
\end{proof}

\begin{lemma}\label{lem:orient}
Let $h : K_s \to K_s$ be a homeomorphism that is pseudo-isotopic to the identity. Then $h$ preserves
orientation of every arc-component $\A0$, \ie given a parametrization $\phi : \R \to \A0$ (or
$\phi : [0,1] \to \A0$ or $\phi : [0,\infty) \to \A0$) that induces an order $\prec$ on $\A0$,
then $x \prec y$ implies $h(x) \prec h(y)$.
\end{lemma}

\begin{proof}
Let us first suppose that $h:K_s \to K_s$ is any homeomorphism. Then, by \cite[Theorem 1.2]{BBS} there is
an $R \in \Z$ such that $h$, restricted to the core, is pseudo-isotopic to $\sigma^R$, \ie $h$ permutes the
composants of the core of the inverse limit in the same way as $\sigma^R$. Therefore, by Corollary \ref{cor:arc-components}, it permutes the arc-components of the inverse limit in the same way as $\sigma^R$.

Let $\A0, \A0'$ be arc-components of the core such that $h, \sigma^R : \A0 \to \A0'$, and let $x, y \in \A0$,
$x \prec y$. We want to prove that $h(x) \prec h(y)$ if and only if $\sigma^R(x) \prec \sigma^R(y)$. Since $h$
and $\sigma^R$ are homeomorphisms on arc-components, each of them could be either order preserving or order
reversing. Therefore, to prove the claim we only need to pick two convenient points $u, v \in \A0$,
$u \prec v$, and check if we have either $h(u) \prec h(v)$ and $\sigma^R(u) \prec \sigma^R(v)$, or
$h(v) \prec h(u)$ and $\sigma^R(v) \prec \sigma^R(u)$. If $\A0$ contains at least two folding points, we can
choose $u, v$ to be folding points. Then $h(u) = \sigma^R(u)$ and $h(v) = \sigma^R(v)$ and the claim follows.

Let us suppose now that $\A0$ contains at most one folding point. Then there exist $q, p \in \N_0$ such that
$h(\chain_{q}) \preceq \chain_{p}$ and $q$-points $u, v \in \A0$, $u \prec v$ (on the same side of the folding
point if there exists one) such that $\sigma^R(u)$ and $\sigma^R(v)$ are contained in disjoint links of
$\chain_{p}$ each of which does not contain the folding point of $\A0$, if there exists one.

Let $\ell_p^{s_j}, \ell_p^{s_k} \in \chain_{p}$ with $j = L_p(\sigma^R(u))$ and
$k = L_p(\sigma^R(v))$ be links containing $\sigma^R(u)$ and $\sigma^R(v)$ respectively. Let
$W_u \subset \ell_p^{s_j}$ and $W_v \subset \ell_p^{s_k}$ be arc-components containing $\sigma^R(u)$ and
$\sigma^R(v)$ respectively. Then $W_u$ and $W_v$ do not contain any folding point and by
Lemma~\ref{lem:finite} $h(u) \in W_u$ and $h(v) \in W_v$. Therefore obviously $h(u) \prec h(v)$
if and only if $\sigma^R(u) \prec \sigma^R(v)$.

If $h$ is a homeomorphism that is pseudo-isotopic to the identity, then $R = 0$ and the claim of lemma follows.
\end{proof}

\begin{corollary}\label{cor:arc_A}
If $h$ is pseudo-isotopic to the identity, then the arc $A$ connecting $x$ and $h(x)$ is a single point,
or $A$ contains no folding point.
\end{corollary}

\begin{proof} Since $h$ is pseudo-isotopic to the identity, $x$ and $h(x)$ belong to the same composant,
and in fact the same arc-component. So let $A$ be the arc connecting $x$ and $h(x)$. If $x = h(x)$, then
there is nothing to prove. If $h(x) \neq x$, say $x \prec h(x)$, and $A$ contains a folding point $y$,
then $x \prec y = h(y) \prec h(x)$, contradicting Lemma~\ref{lem:orient}.
\end{proof}

In particular, any homeomorphism $h$ that is pseudo-isotopic to the identity cannot reverse the bar of a
$\sin \frac1x$-continuum, or reverse a {\em double spiral} $S \subset K_s$, see Remark~\ref{rem:folding}.
The next lemma strengthens Lemma~\ref{lem:finite} to the case that $W_x$ is allowed to contain folding points.

\begin{lemma}\label{lem:infinite}
Let $h : K_s \to K_s$ be a homeomorphism that is pseudo-isotopic to the identity. Let $q, p \in \N_0$ be
such that $h(\chain_{q}) \preceq \chain_{p}$. Let $x$ be a $q$-point in the core of $K_s$ and let
$\ell_p^{s_l} \in \chain_{p}$ be such that $l = L_p(x)$. Let $W_x \subset \ell_p^{s_l}$ be an arc-component
of $\ell_p^{s_l}$ containing $x$. Then $h(x) \in W_x$.
\end{lemma}

\begin{proof}
If $W_x$ does not contain any folding point the proof follows by Lemma~\ref{lem:finite} for $R=0$.

Let $W_x$ contain at least one folding point. If $x$ is a folding point, then
$h(x) = x \in W_x$ by Theorem~\ref{thm:folding}.
If $W_x$ contains at least two folding points, say $y$ and $z$, such that
$x \in [y, z] \subset W_x$, then $h(x) \in [y, z] \subset W_x$ by Corollary~\ref{cor:arc_A}.

The last possibility is that $x \in (y, z) \subset W_x$, where $z \in  W_x$ is a folding point,
$y \notin  W_x$, \ie $y$ is a boundary point of $W_x$, and $(y, z)$ does not contain any
folding point. Since $\C0$ is dense in $K_s$, there exists a sequence $(W_i)_{i \in \N}$ of arc-components
of $\ell_p^{s_l}$ such that $W_i \subset \C0$ and $W_i \to (y, z]$ in the Hausdorff metric. Note that for
the sequence of $p$-points $(m_i)_{i \in \N}$, where $m_i$ is the midpoint of $W_i$, we have $m_i \to z$
and $L_p(m_i) \to \infty$. Also, for every $i$ large enough, every $W_i$ contains a $q$-point $x_i$ with
$L_q(x_i) = L_q(x)$, and for the sequence of $q$-points $(x_i)_{i \in \N}$ we have $x_i \to x$. Obviously
$(x_i)_{i \in \N} \subset \C0$ and $L_p(x_i) = L_p(x)$. By the proof of \cite[Proposition 4.2]{BBS} applied
for $R=0$ we have $h(x_i) \in W_i$ for every $i$. Since $h$ is a homeomorphisms, $h(x_i) \to h(x)$. Therefore,
$h(x) \in (y, z) \subset W_x$.
\end{proof}

\begin{proposition}\label{prop:Hausdorff}
Let $h:K_s \to K_s$ be a homeomorphism. If $z^n \to z$ and $A^n = [z^n, h(z^n)]$, then
$A^n \to A := [z, h(z)]$ in Hausdorff metric.
\end{proposition}

\begin{proof}
We know that $h$ is pseudo-isotopic to $\sigma^R$ for some $R \in \Z$; by composing $h$ with $\sigma^{-R}$
we can assume that $R = 0$. By Corollary \ref{cor:arc-components}, $h$ preserves the arc-components, and
by Lemma~\ref{lem:orient}, preserves the orientation of each arc-component as well.

Take a subsequence such that $A^{n_k}$ converges in Hausdorff metric, say to $B$. Since $x, h(x) \in B$,
we have $B \supset A$. Assume by contradiction that $B \neq A$. Fix $q, p$ arbitrary such that $h(\chain_q)$
refines $\chain_p$, and such that $\pi_p(B) \neq \pi_p(A)$ and a fortiori, that there is a link
$\ell \in \chain_p$ such that $\ell \cap A = \emptyset$ and $\pi_p(\ell)$ contains a boundary point of
$\pi_p(B)$.

Let $d_n = \max\{ L_p(y) : y \text{ is $p$-point in } A^n\}$. If $D := \sup d_n < \infty$, then we can pass
to the chain $\chain_{p+D}$ and find that all $A^{n_k}$'s go straight through $\chain_{p+D}$, hence the limit
is a straight arc as well, stretching from $x$ to $h(x)$, so $B = A$. Therefore $D = \infty$, and we can assume
without loss of generality that $d_{n_k} \to \infty$.

Since the link in $\ell$ is disjoint from $A$ but $\pi_p(\ell)$ contains a boundary point of $\pi_p(B)$, the
arcs $A^{n_k}$ intersects $\ell$ for all $k$ sufficiently large. Therefore $A^{n_k} \cap \ell$ separates
$x^{n_k}$ from $h(x^{n_k})$; let $W^{n_k}$ be a component of $A^{n_k} \cap \ell$ between  $x^{n_k}$ and
$h(x^{n_k})$. Since $\pi_p(\ell)$ contains a boundary point of $\pi_p(B)$, $W^{n_k}$ contains at least one
$p$-point for each $k$. Lemma~\ref{lem:infinite} states that there is $y^{n_k} \in W^{n_k}$ such that
$h(y^{n_k}) \in W^{n_k}$ as well, and therefore $x^{n_k} \prec y^{n_k}, h(y^{n_k}) \prec h(x^{n_k})$ (or
$y^{n_k} \prec x^{n_k}, h(x^{n_k}) \prec h(y^{n_k})$), contradicting that $h$ preserves orientation.
\end{proof}

Let us finally prove Theorem~\ref{thm:Fibo_iso}:

\begin{proof}[Proof of Theorem~\ref{thm:Fibo_iso}] Fix $R$ such that $h$ is pseudo-isotopic to $\sigma^R$.
Then $\sigma^{-R} \circ h$ is pseudo-isotopic to the identity. So renaming $\sigma^{-R} \circ h$ to $h$ again,
we need to show that $h$ is isotopic to the identity.

If $x$ is a folding point of $K_s$, then $h(x) = x$ by Theorem~\ref{thm:folding}. In this case, and in fact for
any point such that $h(x) = x$, we let $H(x,t) = x$ for all $t \in [0,1]$. If $h(x) \neq x$, then $x$ and $h(x)$
belong to the same arc-component, and the arc $A = [x, h(x)]$ contains no folding point by
Corollary~\ref{cor:arc_A}. By Lemma~\ref{lem:folding_p-points}, $A$ contains only finitely many $p$-points, so
there is $m$ such that $\pi_m: A \to \pi_m(A)$ is one-to-one. In this case,
$$
H(x,t) = \pi_m^{-1}|_A [ (1-t) \pi_m(x) + t \pi_m(h(x)) ].
$$
Clearly $t \mapsto H( \cdot, t)$ is a family of maps connecting $h$ to the identity in a single path as
$t \in [0,1]$. We need to show that $H$ is continuous both in $x$ and $t$, and that $H( \cdot, t)$ is a bijection
for all $t \in [0,1]$.

Let $z \in K_s$ and $(z^n, t^n) \to (z, t)$. If $h(z) = z$, then $H(z,t) \equiv z$, and
Proposition~\ref{prop:Hausdorff} implies that $H(z^n, t^n) \to z =  H(z,t)$. So let us assume that
$h(z) \neq z$. The arc $A = [z, h(z)]$ contains no folding point, so by Lemma~\ref{lem:folding_p-points},
for all $x \in A$, there is $\eps(x) > 0$ and $W(x) \in \N$ such that $B_{\eps(x)}(x)$ contains no $p$-point of
$p$-level $\geq W(x)$. By compactness of $A$, $\eps := \inf_{x \in A} \eps(x) > 0$ and
$\sup_{x \in A} W(x) < \infty$, whence there is $m > p+W$ such that $V := \pi_m^{-1} \circ \pi_m(A)$ is
contained in an $\eps$-neighborhood of $A$ that contains no $p$-point.

By Proposition~\ref{prop:Hausdorff}, there is $N$ such that $A^n \subset V$ for all $n \geq N$, and in fact
$\pi_m(A^n) \to \pi_m(A)$. It follows that $H(z^n, t^n) \to H(z,t)$.

To see that $x \mapsto H(\cdot ,t)$ is injective for all $t \in [0,1]$, assume by contradiction that there is
$t_0 \in [0,1]$ and $x \neq y$ such that $H(x,t_0) = H(y,t_0)$. Then $x$ and $y$ belong to the same arc-component
$\A0$, which is the same as the arc-component containing $h(x)$ and $h(y)$. The smallest arc $J$ containing all
four point contains no folding point by Corollary~\ref{cor:arc_A}. Therefore there is $m$ such that
$\pi_m:J \to \pi_m(J)$ is injective, and we can choose an orientation on $\A0$ such that $x < y$ on $J$, and
$\pi_m(x) < \pi_m(y)$. Since $t \mapsto \pi_m \circ H(x , t)$ is monotone with constant speed depending only on
$x$, we find
$$
\pi_m(x) < \pi_m(y) <  \pi_m \circ H(x, t_0) = \pi_m \circ H(y, t_0) <  \pi_m \circ h(y) < \pi_m \circ h(x)
$$
This contradicts that $h$ preserves orientation on arc-components, see Lemma~\ref{lem:orient}.

To prove surjectivity, choose $x \in K_s$ arbitrary. If $h(x) = x$, then $H(x,t)= x$ for all  $t \in [0,1]$.
Otherwise, say if $h(x) > x$, there is $y < x$ in the same arc-component as $x$ such that $h(y) = x$. The map
$t \mapsto H( \cdot , t)$ moves the arc $[y,x]$ continuously and monotonically to $[h(y), h(x)] = [x,h(x)]$.
Therefore, for every $t \in [0,1]$, there is $y_t \in [y,x]$ such that $H(y_t, t) = x$. This proves surjectivity.

We conclude that $H(x,t)$ is the required isotopy between $h$ and the identity.
\end{proof}

\medskip
\noindent
Department of Mathematics,
University of Surrey\\
Guildford, Surrey, GU2 7XH,
United Kingdom\\
\texttt{h.bruin@surrey.ac.uk}\\
\texttt{http://personal.maths.surrey.ac.uk/st/H.Bruin/}

\medskip
\noindent
Department of Mathematics,
University of Zagreb\\
Bijeni\v cka 30, 10 000 Zagreb,
Croatia\\
\texttt{sonja@math.hr}\\
\texttt{http://www.math.hr/}$\sim$\texttt{sonja}

\end{document}